\documentclass{amsart}
\usepackage{amssymb,amsthm,url,tikz,color}
\usetikzlibrary{arrows,shapes,positioning}
\newtheorem{theorem}{Theorem}[section]
\newtheorem{lemma}[theorem]{Lemma}

\newtheorem{conjecture}[theorem]{Conjecture}

\newcommand{\GL}{\mathop{\textrm{GL}}}
\newcommand{\Alt}{\mathop{\textrm{Alt}}}

\newcommand{\Cay}{\mathop{\textrm{Cay}}}
\renewcommand{\wr}{\mathop{\textrm{wr}}}

\newcommand{\Sym}{\mathop{\textrm{Sym}}}

\def\Z#1{{\bf Z}(#1)}
\def\cent#1#2{{\bf C}_{{#1}}({{#2}})}
\def\nor#1#2{{\bf N}_{{#1}}({{#2}})}
\newcommand{\Aut}{\mathop{\mathrm{Aut}}}

\begin{document}

\title[Semiregular elements of large order]{Semiregular elements in cubic vertex-transitive graphs and the restricted Burnside problem}
\author[P. Spiga]{Pablo Spiga}
\address{Pablo Spiga, Dipartimento di Matematica Pura e Applicata,\newline
 University of Milano-Bicocca, Via Cozzi 53, 20126 Milano, Italy} 
\email{pablo.spiga@unimib.it}
\subjclass[2010]{Primary 20B25; Secondary 05E18}

\keywords{vertex-transitive, semiregular}

\begin{abstract}
In this paper,  we prove that the maximal order of a semiregular element in the automorphism group of a cubic vertex-transitive graph $\Gamma$ does not tend to infinity as the number of vertices of $\Gamma$ tends to infinity. This  gives a solution (in the negative) to a conjecture of Peter Cameron, John Sheehan and the author~\cite[Conjecture~$2$]{CSS}.

However, with an application of the positive solution of the restricted Burnside problem,  we show that this conjecture holds true when  $\Gamma$ is either a Cayley graph or an arc-transitive graph.
\end{abstract}

\maketitle

\section{Introduction}

In this paper graphs are finite, connected, with no loops and with no multiple edges. A graph $\Gamma$ is \emph{cubic} if it is regular of valency $3$ and \emph{vertex-transitive} if its automorphism group $\Aut(\Gamma)$ acts transitively on the vertices of $\Gamma$.

A permutation $g$ of a finite set $\Omega$ is said to be {\em semiregular} if $1$ is the only element of the cyclic group $\langle g\rangle$ fixing some point of $\Omega$, that is, in the disjoint cycle decomposition of $g$ all cycles have the same length. By analogy, an automorphism $g$ of a graph $\Gamma$ is semiregular if $g$ is a semiregular permutation in its action on the vertices of  $\Gamma$.

An old conjecture of Maru\v{s}i¡\v{c}, Jordan and Klin asserts that every finite vertex-transitive graph admits
a non-identity semiregular automorphism. This conjecture was first proposed in 1981 by Maru\v{s}i\v{c}~\cite[Problem~$2.4$]{Dragan}, then was posed again by Jordan~\cite{Jordan}, and was finally refined by Klin~\cite{Klin} in $1988$. Now, this wide-open conjecture is known as the \textit{Polycirculant conjecture} and is one of the most interesting old-standing problems in the theory of finite permutation groups. We refer to~\cite{cameronbook,bunch} for more details on this beautiful problem and to~\cite{Ted,GiudiciXu,Klavdija} for some remarkable evidence towards this conjecture.

Maru\v{s}i\v{c} and Scapellato~\cite{Dragan2} have shown that every cubic vertex-transitive graph admits a non-identity semiregular automorphism and hence they have settled the Polycirculant conjecture for the automorphism group of a cubic vertex-transitive graph. In their nice argument  there is little information on the order of the semiregular automorphisms. In fact, their proof allows the existence of cubic vertex-transitive graphs admitting non-identity semiregular automorphisms of order only $2$. However, empirical evidence suggests that cubic vertex-transitive graphs do have semiregular automorphisms of large order, as the number of vertices of the graph grows. To make this statement precise, Cameron, Sheehan and the author have proposed the following conjecture~\cite[Conjecture~$2$]{CSS}.

\begin{conjecture}\label{conj}{\rm There exists a function $f:\mathbb{N}\to \mathbb{N}$ with $\lim_{n\to \infty}f(n)=\infty$ such that, if $\Gamma$ is a cubic vertex-transitive graph with $n$ vertices,  then $\Gamma$ has a semiregular automorphism of order at least $f(n)$.}
\end{conjecture}

A short answer to this  open problem is  ``no".
\begin{theorem}\label{neg}
There exists no function $f$ satisfying Conjecture~\ref{conj}.
\end{theorem}
\noindent However, much lies behind this laconic answer and, in fact, we prove that Conjecture~\ref{conj} holds true for a large class of cubic vertex-transitive graphs. 
Let $X$ be a group and let $Y$ be an inverse-closed and identity-free subset of $X$. The \emph{Cayley graph} $\Cay(X,Y)$ has vertex-set $X$ and two vertices $x$ and $x'$ are adjacent if and only if  $x(x')^{-1}\in Y$. We say that a graph $\Gamma$ is a Cayley graph if $\Gamma\cong \Cay(X,Y)$, for some group $X$ and some subset $Y\subseteq X$. Moreover,  a graph $\Gamma$ is said to be \textit{arc-transitive} if $\Aut(\Gamma)$ acts transitively on the ordered pairs of adjacent vertices of $\Gamma$.

The \textit{restricted Burnside problem}~\cite[Section~$1.1$]{VaughanLee} asks whether there exists a function $g:\mathbb{N}\times\mathbb{N}\to \mathbb{N}$ such that, if $G$ is a finite $d$-generated group of exponent $e$, then $|G|\leq g(d,e)$.  In this paper, we use  the positive solution of the  restricted Burnside problem~\cite{Efim1,Efim2} to prove the following. 
\begin{theorem}\label{thrm}
There exists a function $f:\mathbb{N}\to\mathbb{N}$ with $\lim_{n\to \infty}f(n)=\infty$ such that, if $\Gamma$ is a cubic Cayley or arc-transitive graph  with $n$ vertices,  then $\Gamma$ has a semiregular automorphism of order at least $f(n)$.
\end{theorem}
Before concluding this introductory section and moving to the proof of Theorems~\ref{neg} and~\ref{thrm} we show that  Conjecture~\ref{conj} is very much related
to the restricted Burnside problem with $d=3$.  In fact,  for some families of cubic graphs, Conjecture~\ref{conj} is equivalent to the restricted Burnside problem (for the class of finite groups generated either by three involutions or by one involution and one non-involution).

Recall that a graph is called a \emph{graphical regular representation} or \emph{GRR} if its automorphism group acts regularly on its vertices.  Now, if $\Gamma$ is a GRR, then every element of $\Aut(\Gamma)$ is semiregular. In particular, Conjecture~\ref{conj} restricted to a cubic  GRR $\Gamma$ asks whether the maximal element order of $\Aut(\Gamma)$ tends to infinity as the number of vertices of $\Gamma$ tends to infinity. Clearly, as $\Aut(\Gamma)$ acts regularly on the vertices of $\Gamma$, we get that $\Gamma$ has $|\Aut(\Gamma)|$ vertices. Moreover,  a moment's thought gives that the maximal element order of $\Aut(\Gamma)$ tends to infinity if and only if the exponent of $\Aut(\Gamma)$ tends to infinity. Finally, as $\Gamma$ is cubic, the group $\Aut(\Gamma)$ is generated either by three involutions or by one involution and one non-involution. This shows that this particular instance of Conjecture~\ref{conj} is equivalent to the restricted Burnside problem with $d=3$: in fact, the exponent of $\Aut(\Gamma)$ tends to infinity if and only if $|\Aut(\Gamma)|$ tends to infinity.

Now,~\cite[Theorem~$1.2$]{PSV} shows that the class of cubic GRRs is far from being sparse and indeed  McKay and Praeger conjecture~\cite{McKayPraeger} that most cubic vertex-transitive graphs are GRRs. In particular, Conjecture~\ref{conj} for cubic GRRs is a rather interesting and important case.

Observe that ultimately the proof of Theorem~\ref{thrm} relies on the Classification of the Finite Simple Groups. In fact, the Hall-Higman reduction~\cite{HH} of the restricted Burnside problem to the case of prime-power exponent uses the Schreier conjecture. So, since our proof of Theorem~\ref{thrm} is inevitably dependent upon the CFSGs, we feel free to use such a powerful tool in other parts of our argument.

Finally, we do not try to optimize the choice of the function $f$ in Theorem~\ref{thrm} (which we  believe grows rather slowly) and we prove only an ``existence" result. For example, using the recent census of  cubic vertex-transitive graphs~\cite{census}, we see that there exists a cubic Cayley  graph $\Gamma$ with $1152$ vertices and admitting semiregular automorphisms of order at most $6$. For a rather more exotic example see Section~\ref{example}.

We conclude this introductory section, by pointing out the following result that is needed in our proof of Theorem~\ref{thrm} and that might be of independent interest. (A graph $\Gamma$ is said to be $G$-\emph{vertex-transitive} if $G$ is a subgroup of $\Aut(\Gamma)$ acting transitively on the vertex-set $V\Gamma$ of $\Gamma$.)

\begin{theorem}\label{qubiqu}
There exists a function $f:\mathbb{N}\times \mathbb{N}\to \mathbb{N}$ with $$\lim_{n\to \infty}f(n,m)=\infty,\quad\textrm{ for each }m\in \mathbb{N},$$ such that, if  $\Gamma$ is a cubic $G$-vertex-transitive graph with $n$ vertices and $N$ is a minimal normal subgroup of $G$ with $m$ orbit on $V\Gamma$, then $G$ contains a semiregular element of order at least $f(n,m)$. 
\end{theorem}

\subsection{Structure of the paper}The structure of this paper is straightforward. In Section~\ref{example}, we study a cubic graph related to the Burnside group $B(3,6)$, this should highlight once again the relationship between Conjecture~\ref{conj} and the restricted Burnside problem. In Section~\ref{reduction}, we collect some basic results on cubic vertex-transitive graphs and on non-abelian simple groups. In Section~\ref{proofs}, we first prove Theorem~\ref{qubiqu} and then
Theorem~\ref{thrm}. We prove Theorem~\ref{neg} in Section~\ref{sec:construction}.




\section{The group $B(3,6)$ and one interesting example}\label{example}

For positive integers $d$ and $e$, the Burnside group $B(d,e)$ is the freest group on $d$ generators with exponent $e$.  From~\cite{HH}, we see that $B(3,6)$ is actually finite and, in fact, $|B(3,6)|=2^{4375}3^{833}$. This gives that the largest group of exponent $6$ and generated by an involution and by an element of order $6$ is also finite. This group is denoted by $C(2,6)$ in~\cite{C26} and it is shown that $|C(2,6)|=2^{4}\cdot 3^7=34992$. Moreover, a presentation for $C(2,6)$ with two generators $a$ and $b$, where $a^2=b^6=1$, and sixteen relators is given in~\cite[Section~$5$, page~$3633$]{C26}.

Using the computer algebra system \texttt{magma}~\cite{magma}, we have constructed the Cayley graph $\Gamma=\Cay(C(2,6),\{a,b,b^{-1}\})$ and the automorphism group $\Aut(\Gamma)$. In particular, we have checked that $|\Aut(\Gamma)|=2\cdot|C(2,6)|$ and hence $\Aut(\Gamma)$ consists merely of the right regular translations by $C(2,6)$ and of the automorphism $\varphi$ of $C(2,6)$ fixing $a$ and mapping $b$ to $b^{-1}$. With another computation we see that $ab\varphi$ has order $12$ and is semiregular. So, the largest order of a semiregular element of $\Aut(\Gamma)$ is $12$. This gives that in Theorem~\ref{thrm} we have $f(34992)\leq 12$. So, although Theorem~\ref{thrm} shows that  $\lim_{n\to\infty}f(n)=\infty$, we see that the growth rate is conceivably rather slow. 

We are confident that the Burnside group $B(3,6)$ is a rich source of other exotic examples of this form. However, because of the computational complexity, we did not try to construct larger cubic Cayley graphs. 

Finally, we stress again that since we only prove the existence of a function $f$ satisfying Theorem~\ref{thrm}, we do not try (by any means) to obtain the best bounds in the arguments in the next sections.


\section{Some basic results}\label{reduction}
A graph $\Gamma$ is said to be $G$-\textit{arc-transitive} if $G$ is a subgroup of $\Aut(\Gamma)$ acting transitively on the arcs of $\Gamma$, that is, on the ordered pairs of adjacent vertices of $\Gamma$.

Given a vertex $\alpha$ of $\Gamma$, we let $G_\alpha$ be the vertex-stabiliser of $\alpha$ and we let $G_\alpha^{[1]}$ be the pointwise stabiliser of the neighbourhood $\Gamma_\alpha$ of $\alpha$. In particular, $G_\alpha/G_\alpha^{[1]}$ is isomorphic to the permutation group $G_\alpha^{\Gamma_\alpha}$ induced by the action of $G_\alpha$ on $\Gamma_\alpha$. Moreover, for a subgroup $H$ of $G$ we write $\alpha^H$ for the $H$-orbit containing $\alpha$, that is, $\alpha^H=\{\alpha^h\mid h\in H\}$. Finally, given another  vertex $\beta$ of $\Gamma$, we denote by $d(\alpha,\beta)$ the length of a shortest path from $\alpha$ to $\beta$.

 We start our analysis with some preliminary remarks.
\begin{lemma}\label{stabiliserorders}Let $\Gamma$ be a connected cubic $G$-vertex-transitive graph and let $\alpha$ be a vertex of $\Gamma$. Then $G_\alpha^{[1]}$ is a $2$-group. Moreover, either $G_\alpha$ is a $2$-group, or $\Gamma$ is $G$-arc-transitive and $|G_\alpha|=3\cdot 2^s$, for some $s\geq 0$. In particular, every element of $G$ of order coprime to $2$ and $3$ is a semiregular element.
\end{lemma}
\begin{proof}
We start by showing that $G_\alpha^{[1]}$ is a $2$-group. We argue by contradiction and we let $g$ be an element of $G_\alpha^{[1]}$ of prime order $p>2$. Let $\gamma$ be a vertex of $\Gamma$ at minimal distance from $\alpha$ with $\gamma^g\neq \gamma$. Write $d=(\alpha,\gamma)$. As $g$ fixes pointwise $\Gamma_\alpha$, we must have $d\geq 2$. Let $\alpha=\alpha_0,\ldots,\alpha_{d-1},\alpha_d=\gamma$ be a path of length $d$ from $\alpha$ to $\gamma$ in $\Gamma$. By minimality, $\alpha_{d-1}^g=\alpha_{d-1}$ and hence $g$ induces a permutation of the neighbourhood $\Gamma_{\alpha_{d-1}}$ fixing $\alpha_{d-2}$. Since $\Gamma$ is cubic and since  $|g|>2$, we see that $g$ fixes pointwise $\Gamma_{\alpha_{d-1}}$. However, this contradicts $\gamma^g\neq \gamma$.  

Now, $G_\alpha/G_\alpha^{[1]}$ is isomorphic to a subgroup of $\Sym(\Gamma_\alpha)$. As $|\Sym(\Gamma_\alpha)|=6$, the rest of the proof follows easily.
\end{proof}

\begin{lemma}\label{3gen}Let $\Gamma$ be a connected cubic $G$-vertex-transitive graph. Then $G$ contains a $3$-generated subgroup transitive on $V\Gamma$.  Moreover, if $\Gamma$ is also $G$-arc-transitive, then $G$ contains a $6$-generated arc-transitive subgroup. 
\end{lemma}
\begin{proof}
We prove the two parts of this lemma simultaneously.
Fix $\alpha$ a vertex of $\Gamma$. For each $\beta\in \Gamma_\alpha$, choose  $g_\beta\in G$ with $\alpha^{g_\beta}=\beta$. Moreover, if $G$ is arc-transitive, then fix $\beta_0\in \Gamma_\alpha$ and, for each $\beta\in \Gamma_\alpha$, choose $x_\beta\in G_\alpha$ with $\beta_0^{x_\beta}=\beta$. Write $H=\langle g_\beta\mid \beta\in \Gamma_\alpha\rangle$ and, if $G$ is arc-transitive, write also  $K=\langle g_\beta,x_\beta\mid\beta\in \Gamma_\alpha\rangle$. Clearly, $H$ is $3$-generated  and $K$ is $6$-generated . 

Write $\Delta=\alpha^H$. Let $\delta\in \Delta$. So, $\delta=\alpha^h$, for some $h\in H$. As $\Gamma_\alpha\subseteq \Delta$, we get $\Gamma_\delta=\Gamma_{\alpha^h}=\Gamma_\alpha^h\subseteq \Delta^h=\Delta$. Since $\delta$ is an arbitrary element of $\Delta$, this shows that the subgraph induced by $\Gamma$ on $\Delta$ is cubic. As $\Gamma$ is connected, we must have $\Delta=V\Gamma$ and hence $H$ is transitive on  $V\Gamma$. By construction $K_\alpha$ is transitive on $\Gamma_\alpha$ and hence $K$ is arc-transitive.
\end{proof}

Using a celebrated theorem of Tutte~\cite{tutte1} one can (if minded so) replace $6$ by $3$ in the statement of Lemma~\ref{3gen}.
We conclude this section with three lemmas on non-abelian simple groups.

\begin{lemma}\label{easy}There exists a function $\nu:\mathbb{N}\to \mathbb{N}$ with $\lim_{n\to\infty}\nu(n)=\infty$ such that, if $T$ is a non-abelian simple group of order $n$ and $R$ is a $\{2,3\}$-subgroup of $T$, then $|T:R|\geq \nu(n)$.
\end{lemma}
\begin{proof}
We argue by contradiction and we assume that there exists no such function $\nu$. This means that there exist 
\begin{description}
\item[(i)] a positive integer $c$,
\item[(ii)] an infinite family of non-abelian simple groups $\{T_m\}_{m\in\mathbb{N}}$ and 
\item[(iii)]a family $\{R_m\}_{m\in\mathbb{N}}$ of $\{2,3\}$-groups
\end{description}
with $|T_m|<|T_{m+1}|$, $R_m\leq T_m$ and $|T_m:R_m|\leq c$, for every $m\in \mathbb{N}$.

Observe that by the Burnside's $p^\alpha q^\beta$-theorem,  $R_m$ is a proper subgroup of $T_m$. Let $K_m$ be the core of $R_m$ in $T_m$. Since $T_m$ is a non-abelian simple group, we have $K_m=1$. Clearly, $|T_m|=|T_m:K_m|\leq |T_m:R_m|!\leq c!$. However this contradicts $\lim_{m\to\infty}|T_m|=\infty$.
\end{proof}

Given two integers $x$ and $f$ with $x\geq 2$ and $f\geq 1$, a prime $r$ is said to be a {\em primitive prime divisor} for $x^f-1$ if $r$ divides $x^f-1$ and if $r$ is coprime to $x^s-1$, for each $1\leq s<f$. 
\begin{lemma}\label{numbertheory}
Let $x$ and $f$ be integers with $x\geq 2$ and $f\geq 1$. Then $x^f-1$ has no primitive prime divisors if and only if either $(x,f)=(2,6)$ or $(x,f)=(2^y-1,2)$ for some $y\in\mathbb{N}$. If $r$ is a primitive prime divisor for $x^f-1$, then $r\geq f+1$.
\end{lemma}
\begin{proof}
The first part follows from~\cite{zg}. Now, let $r$ be a primitive prime divisor for $x^f-1$. A computation shows that $r$ divides $x^m-1$ if and only if $m$ is a multiple of $f$. As $r$ divides $x^{r-1}-1$ by Fermat's little theorem, we deduce that $f$ divides $r-1$. Thus $r-1\geq f$ and $r\geq f+1$.
\end{proof}

For  the proof of the following lemma we recall the definition of Lie rank of a simple group of Lie type $G$, here we follow~\cite[Sections~$13.1$,~$13.2$]{carter}. Now, $G$  is a group with a $(B,N)$-pair and the Weyl group W of this $(B,N)$-pair is a reflection group  generated by a set of simple reflections. The number of simple reflections is defined to be the Lie rank of $G$. So, for instance, $\mathrm{PSU}_3(q)$ has Lie rank $1$. (The order of the Weyl group $W$ for untwisted groups is given in~\cite[Section~$3.6$]{carter} and for twisted groups is given in~\cite[Section~$13.3$]{carter}.)

\begin{lemma}\label{simplegroups}
There exists a function $\mu:\mathbb{N}\to \mathbb{N}$ with $\lim_{n\to\infty}\mu(n)=\infty$ such that, if $T$ is a non-abelian simple group of order $n$, then $T$ contains an element $t$ of order coprime to $2$ and $3$ and with $|t|\geq \mu(n)$.
\end{lemma}
\begin{proof}
We argue by contradiction and we assume that there exists no such function $\mu$. This means that there exist a constant $c$ and an infinite family of non-abelian simple groups $\{T_m\}_{m\in \mathbb{N}}$ such that, $|T_m|<|T_{m+1}|$ and every element of order coprime to $2$ and $3$ in $T_m$ has order at most $c$, for every $m\in \mathbb{N}$.

Suppose that $\{T_m\}_{m\in \mathbb{N}}$ contains a subsequence $\{T_{m_s}\}_{s\in \mathbb{N}}$ such that, for every $s\in \mathbb{N}$, the group $T_{m_s}$ is isomorphic either to an alternating group of degree $d_s$ or to a group of Lie type of Lie rank $d_s$, and  $\lim_{s\to\infty}d_s=\infty$. 

 Let $s_0\in \mathbb{N}$ with $d_s\geq 5$, for every $s\geq s_0$. Now, we see that $d_s!/2$ divides the order of $T_{m_s}$ (if $T_{m_s}\cong \Alt(d_s)$, then this is clear, and if $T_{m_s}$ is a group of Lie type of Lie rank $d_s$, then $d_s!/2$ divides the order of the  Weyl group of $T_{m_s}$ and hence the order of $T_{m_s}$). From Bertrand's postulate, there exists a prime $p_s$ with $d_s/2< p_s\leq d_s$. Observe that $p_s\notin\{2,3\}$ for every $s\geq s_0$.  In particular, for every $s\geq s_0$, the group $T_{m_s}$ contains an element of order $p_s$  coprime to $2$ and $3$. As $\lim_{s\to\infty}p_s=\lim_{s\to\infty}d_s=\infty$, we contradict our choice of $\{T_m\}_{m\in\mathbb{N}}$. This shows that there exists a constant $r$ such that, for every $m\in \mathbb{N}$, the group $T_m$ is either a sporadic simple group, or an alternating group of degree $\leq r$, or a group of Lie type of Lie rank $\leq r$. In particular, since there are only finitely many sporadic simple groups and alternating groups in the family $\{T_{m}\}_{m\in \mathbb{N}}$, by replacing $\{T_{m}\}_{m\in\mathbb{N}}$ with a proper infinite subfamily if necessary, we may assume that every element in $\{T_m\}_{m\in \mathbb{N}}$ is a group of Lie type of Lie rank $\leq r$. Moreover, since there is only a finite number of possible ranks for the groups in $\{T_m\}_{m\in\mathbb{N}}$,  there exists a positive integer $l$ such that, by replacing $\{T_{m}\}_{m\in\mathbb{N}}$ with a proper infinite subfamily if necessary, every element in $\{T_m\}_{m\in \mathbb{N}}$ is a group of Lie type of Lie rank $l$.

Suppose that $\{T_m\}_{m\in \mathbb{N}}$ contains a subsequence $\{T_{m_s}\}_{s\in \mathbb{N}}$ such that, for every $s\in \mathbb{N}$, the group $T_{m_s}$ is a group of Lie type in characteristic $p_s$, and  $\lim_{s\to\infty}p_s=\infty$. Since $T_{m_s}$ contains an element of order $p_s$, we contradict our choice of $\{T_m\}_{m\in \mathbb{N}}$. This shows that there exists a constant $r$ such that, for every $m\in \mathbb{N}$, the group $T_m$ is a group of Lie type in characteristic $p$ with $p\leq r$. Since there are only finitely many primes $\leq r$, arguing as above, we may assume that every element in $\{T_m\}_{m\in \mathbb{N}}$ is a group of Lie type in characteristic $p$, for a fixed prime $p$.

Observe that there is only a finite number of Lie types. 
So, as usual, we may assume that every element in $\{T_m\}_{m\in \mathbb{N}}$ is of the same Lie type.

Summing up, for every $m\in \mathbb{N}$, we  have $T_m=L_l(p^{f_m})$, for some $f_m\in \mathbb{N}$ (where we denote by $L_l(p^f)$ the group of Lie type $L$, of Lie rank $l$, in characteristic $p$ and  defined over the field $p^f$). Let $m_0\in \mathbb{N}$ such that $f_m\geq 7$, for every $m\geq m_0$.

For $m\in \mathbb{N}$, let $e_m$ be the order of the Schur multiplier of $T_m$ (for each group of Lie type the order of the Schur multiplier can be found in~\cite[Table~$6$, page~xvi]{ATLAS}). Observe that $e_m$ is bounded above by a function of $l$. Now, from~\cite[Table~$6$, page~xvi]{ATLAS}, we see that $(p^{f_m}-1)/e_m$ divides $|T_m|$. For $m\geq m_0$, by Lemma~\ref{numbertheory} there exists a primitive prime divisor $r_m$ for $p^{f_m}-1$ and, moreover, $r_m\geq f_m+1$. As $T_m$ contains an element of order $r_m$ and as $\lim_{m\to\infty}r_m=\lim_{m\to\infty}f_m=\infty$, we obtain a final contradiction.  
\end{proof}

Lemma~\ref{lemma14} is inspired by~\cite[Lemma~$14$]{PrSV} and part of its proof is taken directly from~\cite[Section~$3$]{PrSV} (unfortunately we were unable to readily deduce the statement of Lemma~\ref{lemma14} from~\cite[Lemma~$14$]{PrSV}).
\begin{lemma}\label{lemma14}
Let $T$ be a non-abelian simple group, let $\ell\geq 1$, let $\ell'$ be a divisor of $\ell$, let $S_1,\ldots,S_{\ell'}$ be proper subgroups of $T$
 and let $M$ be an $m$-generated subgroup of $T^\ell$. Suppose that $$T^\ell=M\left((S_1)^{\frac{\ell}{\ell'}}\times \cdots \times (S_{\ell'})^{\frac{\ell}{\ell'}}\right).$$ Then $(\ell/\ell')\leq |T|^m$.
\end{lemma}
\begin{proof}
For $j\in \{1,\ldots,\ell\}$, denote by $T_j$ the $j$th direct factor of $T^\ell$.  Moreover, for each $i\in \{1,\ldots,\ell'\}$, write $$N_i=T_{\frac{(i-1)\ell}{\ell'}+1}\times T_{\frac{(i-1)\ell}{\ell'}+2}\times \cdots \times T_{\frac{i\ell}{\ell'}}.$$
Clearly, $N_i\cong T^{\frac{\ell}{\ell'}}$ and $T^\ell=N_1\times N_2\times\cdots \times N_{\ell'}$. For $i\in \{1,\ldots,\ell'\}$, denote by $\pi_i:T^\ell\to N_i$ the natural projection and write $M_i=\pi_i(M)$. Clearly, $M_i$ is $m$-generated and 
\begin{equation}\label{eq0}
T^{\frac{\ell}{\ell'}}=M_i\,S_i^{\frac{\ell}{\ell'}}.
\end{equation}
This shows that in the proof of this lemma we may assume that $\ell'=1$. 

We argue by contradiction and we suppose that $\ell >|T|^{m}$. Let $g_1,\ldots,g_m$ be $m$ generators of $M$. For each $i\in \{1,\ldots,m\}$, we have $g_i=(t_{i,1},\ldots,t_{i,\ell})$, for some $t_{i,j}\in T$. Since the entries $\{t_{1,j}\}_{j}$ of the element $g_1$ are in $T$ and since  $T$ has $|T|$ elements, the pigeon-hole principle gives that there exist at least $\ell/|T|$ coordinates on which $g_1$ is constant. In other words, there exists $X_1\subseteq \{1,\ldots,\ell\}$ with $|X_1|\geq \ell/|T|$ and with $t_{1,j}=t_{1,j'}$, for each $j,j'\in X_1$. Arguing in a similar manner with the element $g_2$ and with the coordinates labelled from the elements of $X_1$, we obtain that there exists $X_2\subseteq X_1$ with $|X_2|\geq |X_1|/|T|$ and with $t_{2,j}=t_{2,j'}$, for each $j,j'\in X_2$. In particular, $|X_2|\geq \ell/|T|^2$. Now, an inductive argument gives that there exists a subset $X_m\subseteq \{1,\ldots,\ell\}$ with $|X_m|\geq \ell/|T|^m$ and with $t_{i,j}=t_{i,j'}$, for every $i\in \{1,\ldots,m\}$ and for every $j,j'\in X_m$.

As we are assuming that $\ell>|T|^m$, we obtain $|X_m|\geq 2$. Relabelling the index set $\{1,\ldots,\ell\}$ if necessary, we may assume that $1,2\in X_m$. Now, consider $\pi:N\to T^2$ the natural projection onto the first two coordinates and consider the diagonal subgroup $D=\{(t,t)\in T^2\mid t\in T\}$ of $T^2$. By construction we have $\pi(M)\leq D$. So, by applying $\pi$ on both sides of~\eqref{eq0}, we obtain 
\begin{equation}\label{eq00}
T^2=\pi(M)(S_1\times S_1)=D(S_1\times S_1).
\end{equation}

Let $t$ be an element of $T\setminus S_1$. From~\eqref{eq00}, we have $(1,t)=(a,a)(b,c)$ for some $a\in T$
and $b,c\in S_1$. This 
yields  $a=b^{-1}\in S_1$ and $t=ac\in S_1$, a
contradiction. This contradiction arose from the assumption
$\ell>|T|^m$. Hence $\ell\leq |T|^{m}$ and the lemma is proved.
\end{proof}

\section{Proof of Theorems~\ref{thrm} and~\ref{qubiqu}}\label{proofs}

\begin{proof}[Proof of Theorem~\ref{qubiqu}]
Without loss of generality, we may assume our graphs to be connected. So, 
let $\Gamma$ be a connected cubic $G$-vertex-transitive graph with $n$ vertices and let $N$ be a minimal normal subgroup of $G$ with $m$ orbits on $V\Gamma$.  

Let $\mathcal{O}_1,\ldots,\mathcal{O}_m$ be the orbits of $N$ on $V\Gamma$. We show that $\Gamma$ contains $m$ vertices $\beta_1,\ldots,\beta_m$, with $\beta_i\in \mathcal{O}_i$ for every $i$, such that the subgraph induced by $\Gamma$ on $\{\beta_1,\ldots,\beta_m\}$ is connected. Let $X$ be a subset of vertices of $\Gamma$ of maximal size with the properties  
\begin{description}
\item[(i)] $|X\cap \mathcal{O}_i|\leq 1$ for every $i\in \{1,\ldots,m\}$, and
\item[(ii)] the subgraph induced by $\Gamma$ on $X$ is connected. 
\end{description}
If $|X|=m$, then the claim is proved. Suppose then that $|X|=l<m$. Without loss of generality we may assume that $X=\{\beta_1,\ldots,\beta_l\}$. Let $\gamma$ be a vertex in $\mathcal{O}_{l+1}$. Since $\Gamma$ is connected, there exists a path $\beta_1=\gamma_1,\ldots,\gamma_u=\gamma$ in $\Gamma$ from $\beta_1$ to $\gamma$. Let $i$ be minimal such that $$\gamma_i\notin \bigcup_{j=1}^l\mathcal{O}_j.$$ In particular, $i\geq 2$ and $\gamma_{i-1}\in \mathcal{O}_{k}$ for some $k\leq l$. Since $N$ is transitive on $\mathcal{O}_{k}$, there exists $g\in N$ such that $\beta_{k}=\gamma_{i-1}^g$. So, $\gamma_i^g$ is adjacent to $\beta_{k}$ and $\gamma_i^g\notin \cup_{j=1}^l\mathcal{O}_j$. Set $X'=X\cup\{\gamma_i^g\}$. By construction, $X\subset X'$, the subgraph induced by $\Gamma$ on $X'$  is connected and $X'$ contains at most one vertex from each $\mathcal{O}_i$. This contradicts the maximality of $X$. Thus $|X|=m$ and the claim is proved.

Fix $\beta_1,\ldots,\beta_m$, with $\beta_i\in \mathcal{O}_i$ for every $i\in \{1,\ldots,m\}$, such that the subgraph induced by $\Gamma$ on $\{\beta_1,\ldots,\beta_m\}$ is connected.

For each $i\in \{1,\ldots,m\}$, write $\Gamma_{\beta_i}=\{\beta_{i,1},\beta_{i,2},\beta_{i,3}\}$. Now, for each $i\in \{1,\ldots,m\}$ and $j\in \{1,2,3\}$, there exists a unique $k_{i,j}\in \{1,\ldots,m\}$ with $\beta_{i,j}\in \mathcal{O}_{k_{i,j}}$. So, choose $g_{i,j}\in N$ with $\beta_{i,j}^{g_{i,j}}=\beta_{k_{i,j}}$ and set 
$$
M=\langle g_{i,j}\mid i\in \{1,\ldots,m\},j\in \{1,2,3\}\rangle.$$
Observe that $M$ is $3m$-generated.

We claim that the orbits of $M$ on $V\Gamma$ are exactly $\mathcal{O}_1,\ldots,\mathcal{O}_m$. Write $\Delta=\beta_1^{M}\cup\cdots \cup\beta_m^{M}$. Observe that from the definition of $M$, we have $\Gamma_{\beta_i}\subseteq \Delta$, for each $i\in \{1,\ldots,m\}$. From this it follows that $\Gamma_\delta\subseteq \Delta$, for each $\delta\in \Delta$. Now, recalling that the subgraph induced by $\Gamma$ on $\{\beta_1,\ldots,\beta_m\}$ is connected, it follows from a connectedness argument that the subgraph induced by $\Gamma$ on $\Delta$ is cubic. As $\Gamma$ is itself cubic, we must have $V\Gamma=\Delta$, from which our claim follows. We obtain 
\begin{equation}\label{eq1}
N=MN_\gamma,\quad\textrm{for every }\gamma\in V\Gamma.
\end{equation}

As $N$ is a minimal normal subgroup of $G$, we have $N=T_1\times \cdots \times T_\ell$, where $\ell\geq 1$ and where $T_1,\ldots,T_\ell$  are pairwise isomorphic simple groups (here $T_i$ is possibly abelian).

Suppose that $N$ is soluble, that is, $T_i$ is cyclic of prime order $p$, for some prime $p$. Assume that $N_\gamma=1$, for $\gamma\in V\Gamma$.  Thus $M=N$  and hence $N$ is $3m$-generated. Clearly, this gives $\ell\leq 3m$ and hence  $|N|\leq p^{3m}$. In particular, $n=|V\Gamma|\leq mp^{3m}$. As the elements of $N$ are semiregular, we see that $G$ contains a semiregular element of order $p$. Finally, since $p\geq (n/m)^{\frac{1}{3m}}$ and since $\lim_{n}(n/m)^{\frac{1}{3m}}=\infty$, the lemma follows. 

Assume that $N_\gamma\neq 1$. So, $p\in \{2,3\}$ by Lemma~\ref{stabiliserorders}. Now $|N:N_\gamma|\leq |M|\leq p^{3m}\leq 3^{3m}$. In particular, $\Gamma$ has order at most $m\cdot 3^{3m}$. As the number of vertices of $\Gamma$ is bounded above by a function of $m$, there is nothing to prove in this case.

Suppose that $N$ is not soluble, that is, $T_i$ is a non-abelian simple group, for each $i\in \{1,\ldots,m\}$. Denote by $\pi_i:N\to T_i$ the natural projection onto the $i$th direct factor of $N$. 
Denote by $K$ the kernel of the action of $G$ on $N$-orbits, that is, $K=\cap_{\alpha\in V\Gamma}(G_\alpha N)$. In particular,
$$K=K_\gamma N,\quad \textrm{for every  }\gamma\in V\Gamma.$$
As $N$ has $m$ orbits on $V\Gamma$, we have $|G:K|\leq m!$.

Since $N$ is a minimal normal subgroup of $G$, the group $G$ acts transitively by conjugation on the set of the simple direct summands $\{T_1,\ldots,T_\ell\}$ of $N$. Moreover, as $K\unlhd G$ and as $|G:K|\leq m!$, we obtain that $K$ has at most $m!$ orbits on $\{T_1,\ldots,T_\ell\}$. Denote by $\ell'$ the number of $K$-orbits on $\{T_1,\ldots,T_\ell\}$. So, $\ell'\leq m!$.  Moreover, observe that as $K=K_\gamma N$ (for $\gamma\in V\Gamma$) and as $N$ acts trivially by conjugation on $\{T_1,\ldots,T_\ell\}$, we have that $K$ and $K_\gamma$ have exactly the same orbits on $\{T_1,\ldots,T_\ell\}$. 

Now, fix $\alpha\in V\Gamma$ and write $S_i=\pi_i(N_\alpha)$. Since $N_\alpha$ is a $\{2,3\}$-group, so is $S_i$. From the Burnside's $p^\alpha q^\beta$-theorem, we have $S_i\neq T$ and hence $S_i$ is a  proper subgroup of $T$. Clearly, $N_\alpha\leq S_1\times \cdots \times S_\ell$. Since $N_\alpha\unlhd K_\alpha$, we see that $K_\alpha$ acts by conjugation on the set $\{S_1,\ldots,S_\ell\}$. Moreover, as  $K_\alpha$ has $\ell'$ orbits on $\{T_1,\ldots,T_\ell\}$ each of size $\ell/\ell'$, we get that $K_\alpha$ has exactly $\ell'$ orbits on $\{S_1,\ldots,S_\ell\}$ each of size $\ell/\ell'$. Let $S_{i_1},\ldots,S_{i_{\ell'}}$ be representatives for the orbits of $K_\alpha$ on $\{S_1,\ldots,S_\ell\}$. Thus we have
$$S_1\times \cdots \times S_\ell\cong (S_{i_1})^{\frac{\ell}{\ell'}}\times \cdots\times (S_{i_{\ell'}})^{\frac{\ell}{\ell'}}.$$
Observe further that as $N\cong T_1^\ell$ and $\Aut(N)\cong \Aut(T_1)\wr \Sym(\ell)$, the above isomorphism is induced by an automorphism $\eta$ of the whole of $N$.

From~\eqref{eq1}, we have $$N=N^\eta=(MN_\alpha)^\eta=M^\eta N_\alpha^\eta\leq  M^\eta\left((S_{i_1})^{\frac{\ell}{\ell'}}\times \cdots\times 
(S_{i_{\ell'}})^{\frac{\ell}{\ell'}}\right)=N$$
and we are in the position to apply Lemma~\ref{lemma14} (with $M$ replaced by $M^\eta$, $m$ replaced by $3m$, and $S_1,\ldots,S_{\ell'}$ replaced by $S_{i_1},\ldots,S_{\ell'}$). We obtain $(\ell/\ell')\leq |T|^{3m}$. Recalling that $\ell'\leq m!$, we get $\ell\leq |T|^{3m}m!$.

Observe that each $N$-orbit has size at most $|N|=|T|^\ell$ and at least $|N:R^\ell|=|T:R|^\ell$, where $R$ is a $\{2,3\}$-subgroup of maximal size of $T$. By Lemma~\ref{easy}, we get
$$m\nu(|T|)^{\ell}\leq n\leq m|T|^\ell.$$
As $\ell$ is bounded above by a function of $|T|$ and $m$ only, we see that the order of $\Gamma$ is trapped between two functions depending only on $|T|$ and $m$. So, the proof follows from Lemma~\ref{simplegroups}.
\end{proof}

In view of Theorem~\ref{qubiqu}, for proving Theorem~\ref{thrm} it suffices to remove the dependency on $m$ in the function $f$.

We now recall the important definition of a normal quotient of a graph. Let $\Gamma$ be a $G$-vertex-transitive graph and let $N$ be a normal subgroup of $G$. Let $\alpha^N$ denote the $N$-orbit containing $\alpha\in V\Gamma$. Then the \textit{normal quotient} $\Gamma/N$  is the graph whose vertices are the $N$-orbits on $V\Gamma$, with an edge between distinct vertices $\alpha^N$ and $\beta^N$ if and only if there is an edge $\{\alpha',\beta'\}$ of $\Gamma$ for some $\alpha'\in \alpha^N$ and some $\beta'\in \beta^N$. Observe that the kernel of the action of $G$ on $N$-orbits is $K=\cap_{\alpha\in V\Gamma}G_\alpha N$. Moreover, $\Gamma/K=\Gamma/N$. The group $G/K$ acts faithfully and transitively on the graph $\Gamma/N$ with vertex-stabilisers $G_\alpha K/K$, for $\alpha\in V\Gamma$. Clearly, if $\Gamma$ is $G$-arc-transitive, then $\Gamma/N$ is $(G/K)$-arc-transitive.

\begin{proof}[Proof of Theorem~\ref{thrm}]
We argue by contradiction and we assume this theorem to be false. This means that there exist a constant $c_1$ and an infinite family of cubic vertex-transitive graphs $\{\Gamma_m\}_{m\in \mathbb{N}}$ such that 
\begin{description}
\item[(i)]$\Gamma_m$ is either a Cayley or an arc-transitive graph, 
\item[(ii)]$|V\Gamma_m|<|V\Gamma_{m+1}|$ and  
\item[(iii)]every semiregular element of $\Aut(\Gamma_m)$ has order at most $c_1$,
\end{description}
for every $m\in\mathbb{N}$. Clearly, we may assume that each $\Gamma_m$ is connected.

Suppose that there exists an infinite subsequence $\{\Gamma_{m_s}\}_{s\in\mathbb{N}}$ of Cayley graphs. For each $s\in \mathbb{N}$, let $X_s$ be a finite group and let $Y_s$ be a finite subset of $X_s$ with $\Gamma_{m_s}=\Cay(X_s,Y_s)$. Now, every element of $X_s$ has order at most $c_1$ and hence $X_s$ has exponent at most $c_1!$. By Lemma~\ref{3gen}, $X_s$ is $3$-generated and hence the positive solution of the restricted Burnside problem gives a constant $c_2$ with $|X_s|\leq c_2$, for every $s\in \mathbb{N}$. However, this contradicts the infinitude of $\{\Gamma_{m_s}\}_{s\in \mathbb{N}}$. So, by replacing $\{\Gamma_m\}_{m\in\mathbb{N}}$ with a proper infinite subfamily if necessary, we may assume that $\Gamma_m$ is arc-transitive, for every $m\in \mathbb{N}$.

For $m\in \mathbb{N}$, write $A_m=\Aut(\Gamma_m)$. By Lemma~\ref{3gen}, the group $A_m$ contains a $6$-generated arc-transitive subgroup $G_m$. Let $N_m$ be a maximal (with respect to inclusion) normal subgroup of $G_m$ with $\Gamma_m/N_m$ cubic.  Define 
$\Delta_m=\Gamma_m/N_m$ and $H_m=G_m/N_m$. Observe that by the maximality of $N_m$, the group $H_m$ acts faithfully on $V\Delta_m$ and hence $\Delta_m$ is a cubic $H_m$-arc-transitive graph.   Moreover, as $\Gamma_m/N_m$ is cubic, it follows from an easy connectedness argument that the normal subgroup $N_m$ acts semiregularly on $V\Gamma_m$.

Suppose that there exists a constant $c_2$ with $|V\Delta_m|<c_2$, for every $m\in \mathbb{N}$. 
Since $G_m$ is $6$-generated and since $|G_m:N_m|=|H_m|\leq |\Sym(V\Delta_m)|\leq c_2!$, we see from the Reidemeister-Schreier theorem~\cite[6.1.8~(ii)]{rob} that $N_m$ is $(5c_2!+1)$-generated. As every semiregular element of $G_m$ has order at most $c_1$ and as $N_m$ acts semiregularly on $V\Gamma$, we see that $N_m$ has exponent dividing $c_1!$. In particular, the number of generators and the exponent of $N_m$ are both bounded above by constants. So, from the positive solution of the restricted Burnside problem, we see that there exists a constant $c_3$ with  $|N_m|\leq c_3$. 
For every $m$ in $\mathbb{N}$, we have $|V\Gamma_m|=|V\Delta_m| |N_m|\leq c_2c_3$, a contradiction. This shows that there exists a subsequence $\{m_s\}_{s\in \mathbb{N}}$ with $|V\Delta_{m_s}|<|V\Delta_{m_{s+1}}|$, for every $s\in \mathbb{N}$.

For each $s\in \mathbb{N}$, let  $M_{m_s}/N_{m_s}$ be a minimal normal subgroup of $G_{m_s}/N_{m_s}$. Write $U_{m_s}=M_{m_s}/N_{m_s}$. Now, the maximality of $N_{m_s}$ yields that $$\Gamma_{m_s}/M_{m_s}\cong \Delta_{m_s}/U_{m_s}$$
  is not cubic. Since $\Delta_{m_s}$ is $H_{m_s}$-arc-transitive, we must have that $U_{m_s}$ has at most two orbits on the vertices of $\Delta_{m_s}$. So, by Theorem~\ref{qubiqu}, $H_{m_s}$ contains a semiregular element $g_{m_s}N_{m_s}$ of order $\geq f(|V\Delta_{m_s}|,2)$. Observe that $g_{m_s}$ acts semiregularly on $V\Gamma_{m_s}$ because $N_{m_s}$ acts semiregularly on $V\Gamma_{m_s}$. Thus $G_{m_s}$ contains a semiregular element of order $\geq f(|V\Delta_{m_s}|,2)$. As $\lim_s|V\Delta_{m_s}|=\infty$, for $s$ sufficiently large $G_{m_s}$ contains a semiregular element of order $>c_1$, a contradiction. This finally proves the theorem.
\end{proof}


\section{Proof of Theorem~\ref{neg}}\label{sec:construction}

Let $m\geq 1$ be an integer. We denote by $\mathbb{F}_3$ the field with $3$ elements and we let $J$ be the $(2^{m}\times 2^{m})$-matrix with coefficients in $\mathbb{F}_3$ defined by
\begin{equation}\label{J}
J_{i,j}=
\left\{
\begin{array}{ccl}
(-1)^{i-j}&&\textrm{if }i>j,\\
-(-1)^{j-i}&&\textrm{if }j>i,\\
0&&\textrm{if }i=j.\\
\end{array}
\right.
\end{equation}
Clearly, the matrix $J$ is antisymmetric, that is, $J_{i,j}=-J_{j,i}$ for every $i,j\in \{1,\ldots,2^{m}\}$. 

We show that $J$ is non-degenerate by performing Gaussian elimination.  Keeping the first row of $J$ and, for each $i\in \{2,\ldots,2^{m}\}$, replacing the $i$th row with the sum of the $(i-1)$th and of the $i$th row  of $J$, we obtain the matrix
\[
J'=\left(
\begin{array}{cccccccccc}
 0 & 1  &-1 &  1   &-1&1& &\cdots&1\\
-1 & 1  & 0 &  0   &0& 0& &\cdots&0\\ 
 0 & -1 & 1 &  0   &0&0& &\cdots&0\\
 0 & 0  & -1 &  1  &0&0& &\cdots&0\\
    &     &     &    &   &  &  &          &\\
    &     &     &      & \cdots&0& -1 &    1 &0\\
    &     &     &      &   & \cdots & 0 &    -1 &1\\
\end{array}
\right).
\] 
Observe that the last $2^{m}-1$ rows of $J'$ are linearly independent. Moreover, the entries in the first row of $J'$ add up to $1$, whilst the entries in the other rows of $J'$ add up to $0$. Thus $J'$  is non-degenerate and so is $J$.

Since $J$ is antisymmetric and non-degenerate,   $J$ determines a non-degenerate symplectic form on the $2^{m}$-dimensional vector space $\mathbb{F}_3^{2^{m}}$. We use this bilinear form to construct an extraspecial $3$-group.

Let $V$ be the group given by the presentation
\begin{eqnarray}\label{V}
V=\big\langle v_1,\ldots,v_{2^{m}},z&\mid& v_i^3=z^3=[v_i,z]=1,\\\nonumber
&&[v_i,v_j]=z^{J_{i,j}},\,\textrm{for every }i,j\in \{1,\ldots,2^{m}\}\big\rangle.
\end{eqnarray}
Observe that $V$ is an extraspecial $3$-group of exponent $3$ and of order $3^{2^{m}+1}$.

Let $a:\{v_1,\ldots,v_{2^m},z\}\to V$ and $b:\{v_1,\ldots,v_{2^m},z\}\to V$ be the maps  defined by
\begin{eqnarray}\label{a}
v_i\mapsto v_i^{a}&=&\left\{
\begin{array}{ccl}
v_{i+1}&&\textrm{if }i\neq 2^{m},\\
v_1^{-1}&&\textrm{if }i=2^{m},
\end{array}
\right.\quad\textrm{and}\quad
z\mapsto z^a\,=\,z,
\end{eqnarray}
\begin{eqnarray}\label{b}
v_i\mapsto v_i^{b}&=&\left\{
\begin{array}{ccl}
v_{2^{m}-i+2}&&\textrm{if }i\neq 1,\\
v_{1}^{-1}&&\textrm{if }i=1,\\
\end{array}
\right.\quad\textrm{and}\quad
z\mapsto z^b\,=\,z^{-1}.
\end{eqnarray}
We show that $a$ and $b$ preserve the defining relations~\eqref{V} of $V$ and hence they naturally extend to two automorphisms of $V$, which we will still denote by $a$ and $b$.

  Since both $a$ and $b$ map $z$ to $z^{\pm 1}$ and $v_i$ to an element of the form $v_j^{\pm 1}$, we see that $$(v_i^a)^3=(v_i^b)^3=(z^a)^3=(z^b)^3=[v_i^a,z^a]=[v_i^b,z^b]=[v_i^a,v_i^a]=[v_i^b,v_i^b]=1,$$for every $i\in \{1,\ldots,2^{m}\}$. It remains to show that $[v_i^a,v_j^a]=(z^a)^{J_{i,j}}$ and $[v_i^b,v_j^b]=(z^b)^{J_{i,j}}$, for every two distinct elements $i,j\in \{1,\ldots,2^{m}\}$.
Observe that,  as the commutator satisfies the identity $[x,y]=[y,x]^{-1}$, we may assume that $i>j$. Furthermore, as $V$ is a nilpotent group of class $2$, we have $[x,y^{-1}]=[x^{-1},y]=[x,y]^{-1}$, for every $x,y\in V$. 

We start by dealing with $a$.
If $i\neq 2^{m}$, then using~\eqref{J} and~\eqref{a} we obtain $$
[v_i^a,v_j^a]=[v_{i+1},v_{j+1}]=z^{J_{i+1,j+1}}=z^{(-1)^{(i+1)-(j+1)}}\\
=z^{(-1)^{i-j}}=z^{J_{i,j}}=(z^a)^{J_{i,j}}.$$
If $i=2^{m}$, then, using again~\eqref{J} and~\eqref{a}, we get 
\begin{eqnarray*}
[v_i^a,v_j^a]&=&[v_{1}^{-1},v_{j+1}]=[v_{1},v_{j+1}]^{-1}=\left(z^{J_{1,j+1}} \right)^{-1}\\
&=&z^{-J_{1,j+1}}=z^{(-1)^{j}}=z^{J_{i,j}}=(z^a)^{J_{i,j}}.
\end{eqnarray*}
Now we consider $b$. If $j\neq 1$, then using~\eqref{J} and~\eqref{b} we have
\begin{eqnarray*}
[v_i^b,v_j^b]&=&[v_{2^{m}-i+2},v_{2^{m}-j+2}]=z^{J_{2^{m}-i+2,2^{m}-j+2}}\\
&=&z^{-(-1)^{(2^{m}-i+2)-(2^{m}-j+2)}}=z^{-(-1)^{j-i}}=z^{J_{j,i}}=z^{-J_{i,j}}=(z^b)^{J_{i,j}}.
\end{eqnarray*}
If $j= 1$, then, using again~\eqref{J} and~\eqref{b}, we see
\begin{eqnarray*}
[v_i^b,v_j^b]&=&[v_{2^m-i+2},v_{1}^{-1}]=[v_{2^m-i+2},v_{1}]^{-1}=\left( z^{J_{2^{m}-i+2,1}} \right)^{-1}\\
&=&\left( z^{(-1)^{i-1}} \right)^{-1}=\left(z^{J_{i,j}}\right)^{-1}=z^{-J_{i,j}}=(z^b)^{J_{i,j}}.
\end{eqnarray*}
Our claim is now proved.

Using~\eqref{a}, we see that 
\begin{equation}\label{apower}
v_i^{a^{2^{m}}}=v_i^{-1},\,\,\,\,\,\, \textrm{for every }\, i\in \{1,\ldots,2^{m}\}.
\end{equation}
In particular, $a$ is an automorphism of $V$ of order $2^{m+1}$. By~\eqref{b}, the element $b^2$ fixes every generator of $V$ and hence $b$ is an automorphism of $V$ of order $2$.  Moreover, combining~\eqref{a} and~\eqref{b}, we get
\begin{eqnarray*}
z^{bab}&=&(z^{-1})^{ab}=(z^{-1})^b=z=z^{a^{-1}},\\
v_1^{bab}&=&(v_1^{-1})^{ab}=(v_1^{ab})^{-1}=(v_2^b)^{-1}=v_{2^{m}}^{-1}=v_1^{a^{-1}},\\
v_2^{bab}&=&(v_{2^{m}})^{ab}=(v_1^{-1})^b=v_1=v_2^{a^{-1}}
\end{eqnarray*}
and, for $i\in \{3,\ldots,2^{m}\}$,
\begin{eqnarray*}
v_i^{bab}&=&(v_{2^{m}-i+2})^{ab}=(v_{2^{m}-i+3})^b=(v_{2^{m}-(i-1)+2})^b=v_{i-1}=v_i^{a^{-1}}.
\end{eqnarray*}
It follows that $bab=a^{-1}$. In particular, 
\begin{equation*}
Q=\langle a,b \rangle
\end{equation*}
is isomorphic to a dihedral group of order $2^{m+2}$.

Now, the group $Q$ acts as a group of automorphisms on the vector space $V/\langle z\rangle$ and hence we may regard $V/\langle z\rangle$ as a $Q$-module over $\mathbb{F}_3$.
\begin{lemma}\label{Qirr}The group $Q$ acts irreducibly on $V/\langle z\rangle$.
\end{lemma}
\begin{proof}
If $m=1$ or $m=2$, then the lemma follows from a direct computation. Assume that $m\geq 3$. Write $W=V/\langle z\rangle$. We use an additive notation for the elements of $W$. 
Let $U$ be an irreducible $\langle a\rangle$-submodule of $W$ and let $\ell$ be the dimension of $U$ over $\mathbb{F}_3$. Observe that since $a^{2^m}$ acts as the inversion on $W$ by~\eqref{apower}, we see that $a^{2^{m}}$ acts faithfully on $U$ and so does $a$. Hence by~\cite[$9.4.3$]{rob}, we get that $2^{m+1}=|a|$ divides $3^\ell-1$ and that $\ell$ is the smallest positive integer with $3^\ell\equiv 1\mod 2^{m+1}$. Now, using the binomial expansion $3^{\ell}=(1+2)^{\ell}=\sum{\ell\choose i}2^i$ and using $m\geq 3$, a computation shows that 
\begin{eqnarray*}
3^{2^{m-1}}&\equiv&1\mod 2^{m+1},\\
3^{2^{m-2}}&\equiv&1+2^m\mod 2^{m+1}.
\end{eqnarray*} Therefore $2^{m-1}=\ell$ and hence $\dim U=2^{m-1}$.

By~\eqref{apower} the element $a^{2^m}$ acts by inverting each element of $W$ and hence the characteristic polynomial of $a$ in its action on $W$ is $T^{2^m}+1$. Since the characteristic of $W$ is $3$, we obtain 
\begin{equation}
T^{2^m}+1=(T^{2^{m-1}}+T^{2^{m-2}}-1)(T^{2^{m-1}}-T^{2^{m-2}}-1).
\end{equation}
Write 
\begin{eqnarray*}
W_+&=&\{w\in W\mid w^{a^{2^{m-1}}+a^{2^{m-2}}-1}=0\},\\ 
W_{-}&=&\{w\in W\mid w^{a^{2^{m-1}}-a^{2^{m-2}}-1}=0\}.
\end{eqnarray*}
Clearly, $W_{+}$ and $W_-$ are $\langle a\rangle$-invariant subspaces of $W$ with $\dim W_+=\dim W_-=2^{m-1}$. From the previous paragraph, we see that $W_+$ and $W_-$ are irreducible $\langle a\rangle$-modules. 

Since the characteristic polynomials of the action $a$ on $W_+$ and $W_-$ are different, we see that $W_+$ and $W_-$ are non-isomorphic $\langle a\rangle$-modules. Moreover, as the order of $\langle a\rangle$   is coprime to $3$, we deduce from Maschke's theorem that $W_+$ and $W_-$ are the only proper $\langle a\rangle$-submodules of $W$.

Let $U$ be a non-zero $Q$-submodule of $W$. In particular, $U$ is a non-zero $\langle a\rangle$-submodule of $W$ and hence $W_+\leq U$ or $W_-\leq U$. We have
\begin{eqnarray*}
0&=&0^b=(a^{2^{m-1}}+a^{2^{m-2}}-1)^b=(a^b)^{2^{m-1}}+(a^b)^{2^{m-2}}-1\\
&=&a^{-2^{m-1}}+a^{-2^{m-2}}-1
\end{eqnarray*}
and, multiplying by $-a^{2^{m-1}}$, we obtain $-1-a^{2^{m-2}}+a^{2^{m-1}}=0$.
This shows that $W_+^b=W_-$. Thus $W=W_+\oplus W_-\leq U$ and hence $W$ is an irreducible $Q$-module.
\end{proof}

Define 
\begin{equation}\label{G}G=V\rtimes Q. 
\end{equation}
In what follows we denote the elements of $G$ as ordered pairs $xv$, with $x\in Q$ and $v\in V$. In particular, for $x,y\in Q$ and $v,w\in V$, we have $(xv)(yw)=(xy)(v^yw)$. Moreover, we identify $Q$ and $V$ with their corresponding isomorphic copies in $G$.

Write $H=\langle a^{2^{m}}\rangle$ and let $\Omega$ be the set of right cosets of $H$ in $G$. Clearly, $H^{v_1}=\langle a^{2^m}v_1^2\rangle$ by~\eqref{apower} and hence $H\cap H^{v_1}=1$. Thus $H$ is core-free in $G$ and the action by right multiplication of $G$ on $\Omega$ endows $G$ of the structure of a permutation group. For the rest of this section we regard $G$ as a subgroup of $\Sym(\Omega)$.

\begin{lemma}\label{semiregularG}A semiregular element of $G$ has order $1$, $2$, $3$ or $6$.
\end{lemma}
\begin{proof}
We argue by contradiction and we let $g$ be a semiregular element of order $\ell$ with $\ell\notin\{1,2,3,6\}$. Since $V$ has exponent $3$, $\ell$ is divisible by $4$. So, replacing $g$ by $g^{\ell/4}$, we may assume that $\ell=4$. Since $Q$ is a Sylow $2$-subgroup of $G$,  replacing $g$ by a suitable $G$-conjugate, we may assume that $g\in Q$.  As $Q$ is a dihedral group, we see that $g^2=a^{2^{m}}\in H$ and hence $g^2$ fixes the point $H$ of $\Omega$, a contradiction. 
\end{proof}

Let $\Gamma$ be the directed graph with vertex set $\Omega$ and with arc set 
\begin{equation}\label{Gamma}
\{(Hg,Habg)\mid g\in G\}\cup
\{(Hg,Hbv_1g)\mid g\in G\}\cup
\{(Hg,Hbv_1^{-1}g)\mid g\in G\}.
\end{equation}
\begin{lemma}\label{1}The graph $\Gamma$ has $2^{m+1}3^{2^m+1}$ vertices. Moreover, $\Gamma$ is connected, undirected, cubic and $G$-vertex-transitive.
\end{lemma}
\begin{proof}
The vertex set of $\Gamma$ is $\Omega$ and $|\Omega|=|G:H|=|Q||V|/|H|=2^{m+1}3^{2^m+1}$.

 Now $ab$, $bv_1$ and $bv_1^{-1}$ have order $2$ because $(ab)^2=abab=aa^{-1}=1$ and $(bv_1^{\pm 1})^2=b^2(v_1^{\pm1})^bv_1^{\pm 1}=(v_1^{\pm 1})^{-1}v_1^{\pm 1}=1$ by~\eqref{b}.  Therefore $(Hab,H)$, $(Hbv_1,H)$ and $(Hbv_1^{-1},H)$ are arcs of  $\Gamma$. This shows that $\Gamma$ is undirected. 

The vertices of $\Gamma$ adjacent to $H$ are of the form $Habh$, $Hbv_1h$ and $Hbv_1^{-1}h$, as $h$ runs through the elements of $H$. Using~\eqref{apower} we get 
\begin{eqnarray*}
(Hab)a^{2^{m}}&=&Haba^{2^{m}}=Ha^{2^{m}}ab=Hab,\\
(Hbv_1)a^{2^{m}}&=&Hba^{2^{m}}(v_1)^{a^{2^{m}}}=Ha^{2^{m}}bv_1^{-1}=Hbv_1^{-1},\\
(Hbv_1^{-1})a^{2^{m}}&=&Hba^{2^{m}}(v_1^{-1})^{a^{2^{m}}}=Ha^{2^{m}}bv_1=Hbv_1.\\
\end{eqnarray*}
Since $H=\langle a^{2^m}\rangle$, the neighbourhood of $H$ is $\{Hab,Hbv_1,Hbv_1^{-1}\}$ and hence $\Gamma$ is cubic. 

The definition of the arc set~\eqref{Gamma} immediately gives that $\Gamma$ is $G$-vertex-transitive. It remains to show that $\Gamma$ is connected. Write $K=\langle ab,bv_1,bv_1^{-1}\rangle$. 
Observe that the elements $ab$, $bv_1$ and $bv_1^{-1}$ map the vertex $H$ to each of its three neighbours. Therefore, the connected component of $\Gamma$ containing $H$ is $\{Hk\mid k\in K\}$. Thus it suffices to prove that $G=K$. Now, $(bv_1)(bv_1^{-1})=b^2v_1^bv_1^{-1}=v_1^{-2}=v_1$ by~\eqref{b}. Thus $v_1\in K$ and hence $b=(bv_1)(v_1)^{-1}\in K$. Now, $a=(ab)b\in K$. Thus $Q\leq K$. Observing that the conjugates of $v_1$ under $Q$ generate the whole of $V$, we get $V\leq K$ and hence $G=K$.
\end{proof}

Write $A=\Aut(\Gamma)$. (As usual, we denote by $A_\alpha$ the stabiliser in $A$ of the vertex $\alpha$ of $\Gamma$. In particular, $A_H$ is the stabiliser of the vertex $H$.)

\begin{lemma}\label{2}$A_H$ is a $2$-group and $|A:G|$ is a $2$-power.
\end{lemma}
\begin{proof}
Since $G$ acts transitively on $V\Gamma$, we have $A=A_HG$ and hence it suffices to show that $A_H$ is a $2$-group. We argue by contradiction. So $A_H$ acts transitively on $\Gamma_H$ by Lemma~\ref{stabiliserorders}. Now,  it is easy to verify that
$$\left(H,Hbv_1,Hv_1^{-1},Hb,Hv_1,Hbv_1^{-1}\right)$$
is a cycle of $\Gamma$ of length $6$ passing through the two neighbours $Hbv_1$ and $Hbv_1^{-1}$ of $H$. As $A_H$ is transitive on $\Gamma_H$, the graph $\Gamma$ admits a cycle of length $6$ passing through the two neighbours $Hab$ and $Hbv_1$ of $H$. 

Using the definition of $\Gamma$ and~\eqref{a} and~\eqref{b}, we compute all the vertices of $\Gamma$ at distance $\leq 2$ from either $H$ or $Hab$. Figure~\ref{finger} depicts this  small neighbourhood of $\Gamma$ and shows that there is no cycle of length $6$ containing $H,Hab$ and $Hbv_1$. (For simplicity, in Figure~\ref{finger} we write $x$ to denote the vertex $Hx$.) This contradiction shows that $\Gamma$ is not arc-transitive.

\begin{figure}[!hhh]
\begin{tikzpicture}[node distance =.7cm]
\tikzset{myarrow/.style={==, thick}}
\node[circle,inner sep=0pt, label=90:$1$](A0){};
\node[right=of A0](A1){};
\node[above=of A1](A2){$bv_1^{-1}$};
\node[left=of A0](A3){};
\node[above=of A3](A4){$bv_1$};
\node[above=of A2](A5){$v_1$};
\node[right=of A5](A6){$av_1^{-1}$};
\node[above=of A4](A7){$v_1^{-1}$};
\node[left=of A7](A8){$av_1$};
\node[below=of A0](B0){$ab$};
\node[below=of B0](B1){};
\node[left=of B1](B2){$a^{-1}v_{2^m}$};
\node[right=of B1](B3){$a^{-1}v_{2^m}^{-1}$};
\node[below=of B2](B4){$abv_{2^m}^{-1}$};
\node[left=of B4](B5){$a^2bv_{2^m}$};
\node[below=of B3](B6){$a^2bv_{2^m}^{-1}$};
\node[right=of B6](B7){$abv_{2^m}$};
\draw(A0) to (B0);
\draw(B0)to (B2);
\draw(B0)to (B3);
\draw(B2)to (B4);
\draw(B2) to (B5);
\draw(B3) to (B6);
\draw(B3)to (B7);
\draw(A0) to (A2);
\draw(A0) to(A4); 
\draw(A2)to(A5);
\draw(A2) to (A6);
\draw(A4) to (A7);
\draw(A4) to (A8);
 \end{tikzpicture}
  \caption{Local structure of $\Gamma$}\label{finger}
\end{figure}
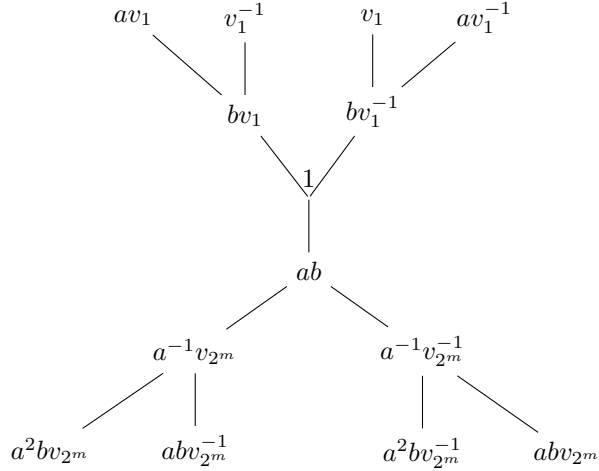
\end{proof}
Before continuing our discussion we need an elementary lemma on the abelian subgroups of the general linear group $\GL_\ell(2)$.
\begin{lemma}\label{ab}
An elementary abelian $3$-subgroup of $\mathrm{GL}_\ell(2)$ has size at most $3^{\ell/2}$.
\end{lemma}
\begin{proof}
Observe $\GL_\ell(2)$ is a group of Lie type. It can be readily checked in~\cite[Theorem~$4.10.3$~(a)]{GLS} 
that the maximal order of an elementary abelian $3$-subgroup of $\GL_\ell(2)$ is 
$3^{\lfloor \ell/2\rfloor }$.
\end{proof}

\begin{lemma}\label{3}Assume that $G<A$ and let $T$ be a minimal (with respect to inclusion) subgroup of $A$ with $G<T$. Then $V$ is normal in $T$.
\end{lemma}
\begin{proof}
Let $K$ be the core of $G$ in $T$. Since $G$ is a maximal subgroup of $T$, the group $T/K$ acts primitively on the set of right cosets of $G$ in $T$. Since $T/K$ is a $\{2,3\}$-group, $T/K$ is of  ``affine" O'Nan-Scott type. Let $S/K$ be the socle of $T/K$. By primitivity and by Lemma~\ref{2}, $S/K$ is an elementary abelian $2$-group with $S\cap G=K$ and $T=SG$. Let $2^\ell$ be the order of $|S/K|$. Now, if $V\leq K$, then $T/K$ is a $2$-group and hence the primitivity forces $|T:G|=2$. As $V$ is characteristic in $G$ and $G\unlhd T$, we get $V\unlhd T$. 

Suppose that $V\nleq K$. We show that this yields to a contraction. Then $V\cap K$ is a normal subgroup of $G$ properly contained in $V$. As $Q$ acts irreducibly on $V/\langle z\rangle$ by Lemma~\ref{Qirr}, we must have $V\cap K=1$ or $V\cap K=\langle z\rangle$.  As $T/K$ is a primitive group of affine type with point stabiliser $G/K$, the group  $G/K$ is isomorphic to an irreducible subgroup of $\GL_\ell(2)$. In particular, $\GL_\ell(2)$ contains either an elementary abelian $3$-group of order $3^{2^m}$ (if $V\cap K=\langle z\rangle$) or an extraspecial $3$-group of order $3^{2^m+1}$ (if $V\cap K=1$). As $V$  contains an elementary abelian $3$-subgroup of order $3^{2^{m-1}+1}$, in both cases we see that $\GL_\ell(2)$ contains an elementary abelian $3$-subgroup of order at least $3^{2^{m-1}+1}$. So, Lemma~\ref{ab} gives 
\begin{equation}\label{eqell}
\ell\geq 2(2^{m-1}+1)=2^{m}+2.
\end{equation}

As $|K\cap V|\leq 3$ and as $V$ is a Sylow $3$-subgroup of $T$, we see that $|S|=3^i 2^{\ell'}$, for some $i\in \{0,1\}$ and some $\ell'\geq \ell$. Let $R$ be the largest normal $2$-subgroup of $S$. Clearly, $R$ is characteristic in $S$ and hence normal in $T$. Moreover, $|S:R|\leq 6$ and hence $|R|\geq 2^{\ell'-1}$. We now distinguish two cases, depending on whether $R$ is semiregular or not.

Assume that $R$ is semiregular. Then $RV$ is a subgroup of $T$ acting semiregularly on $V\Gamma$ because $R$ is semiregular and because $A_H$ is a $2$-group by Lemma~\ref{2}. Thus $$2^{m+1}3^{2^m+1}=|V\Gamma|\geq |RV|=|R||V|=|R|3^{2^m+1}\geq 2^{\ell'-1}3^{2^m+1}$$ and $m+1\geq \ell'-1\geq \ell-1$. However, this contradicts~\eqref{eqell}.

Assume that $R$ is not semiregular. Let $\Delta$ be the normal quotient $\Gamma/R$. As $R$ is not semiregular, we see that $\Delta$ has valency $2$ and hence $\Delta$ is a cycle. Let $F$ be the kernel of the action of $T$ on $R$-orbits. As $R$ and $A_H$ are both  $2$-groups, so is $F$. Moreover, as $\Delta$ is a cycle, $T/F$ is isomorphic to a dihedral group. However, the $3$-group $VF/F\cong V/(V\cap F)=V$ is not isomorphic to a subgroup of a dihedral group,  a contradiction.
\end{proof}

In the next lemma we denote by $\cent X Y$ the centraliser of $Y$ in $X$, by $\nor X Y$ the normaliser of $Y$ in $X$ and by $\Z X$ the centre $X$.
\begin{lemma}\label{4}$A=G$.
\end{lemma}
\begin{proof}
We argue by contradiction and we assume that $G<A$. Let $T$ be a minimal, with respect to inclusion, subgroup of $A$ with $G<T$. From Lemma~\ref{3}, we have $V\unlhd T$. Now, $T/V$ is a $2$-group by Lemma~\ref{2} and hence, by minimality, $|T:G|=2$ and $|T_H|=4$. 

 Write $R=\cent T H $. We show that $R$ is a  Sylow $2$-subgroup of $T$ and that $R=QT_H$. Observe that, since $H$ is the centre of $Q$, we get that  $HV/V$ is the centre of $G/V$. Thus $HV/V$ is a characteristic subgroup of $G/V$ and hence  normal in $T/V$. This shows that $HV\unlhd T$.  Since $H$ is a Sylow $2$-subgroup of $HV$, from the Frattini argument we obtain $T=\nor T H HV=\nor TH V$. As $|H|=2$, we have $\nor T H=\cent T H=R$. Therefore $T=RV$. Clearly, $R\cap V=\cent V H=1$ and hence $R$ is a Sylow $2$-subgroup of $T$. Finally, note that $Q\leq R$ because $H=\Z Q$ and $T_H\leq R$ because $|T_H:H|=|H|=2$. 

Fix $c\in T_H\setminus H$. Observe that replacing $c$ by an element in the coset $Hc$, we may assume that $c$ fixes pointwise the neighbourhood $\Gamma_H=\{Hab,Hbv_1,Hbv_1^{-1}\}$ of $H$. 
Since $c$ normalises $Q$ and since $Q$ is a dihedral group of order $2^{m+1}$, we must have $a^c=a^{i}$ and $b^c=a^jb$, for some odd $i\in \{1,\ldots,2^{m+1}-1\}$ and some $j\in \{0,\ldots,2^{m+1}-1\}$. 

Given a vertex $Hx$ of $\Gamma$, we denote by $(Hx)^c$ the image of $Hx$ under the automorphism $c$. We have
\begin{eqnarray*}
Hab&=&(Hab)^c=H^{abc}=H^{c(c^{-1}abc)}=(H^c)^{(ab)^c}\\
&=&H^{(ab)^c}=H^{a^cb^c}=H^{a^ia^jb}=Ha^{i+j}b
\end{eqnarray*}
and hence $a^{i+j-1}\in H$. Thus $i+j=1$ or $i+j=1+2^m$. Similarly, recalling that $c$ normalises also $V$, we have
\begin{eqnarray*}
Hbv_1&=&(Hbv_1)^c=H^{bv_1c}=H^{c(c^{-1}bv_1c)}=(H^c)^{(bv_1)^c}\\
&=&H^{(bv_1)^c}=H^{b^c v_1^c}=H^{a^jbv_1^c}=Ha^jbv_1^c
\end{eqnarray*}
and hence $a^{j}\in H$ and 
\begin{equation}\label{newv1}
v_1^c=v_1.
\end{equation}
In particular, $j=0$ or $j=2^m$. In what follows we discuss the various possibilities for $i$ and $j$. 

Suppose that $i=1$. Write $C=\cent T c$. As $i=1$, we have $a^c=a$ and hence  $a\in C$. Since $v_1\in C$ and since $V=\langle v_1^{a^k}\mid k\rangle$, we obtain $V\leq C$. Thus $V\langle a,c\rangle\leq C$ and hence $|T:C|\leq 2$. Assume that $T=C$. So, $c$ is a central element of $T$ fixing $H$. As $T$ acts faithfully on $V\Gamma$, we have $c=1$, a contradiction. Assume that $|T:C|=2$, that is, $C=V\langle a,c\rangle$. Now $C$ is intransitive on $V\Gamma$ and the $C$-orbits form a bipartition for $\Gamma$. As $c$ centralises $C$ and fixes the vertex $H$ together with its neighbour $Hab$,  we see that $c$ fixes every vertex in $H^C\cup (Hab)^C=V\Gamma$. Thus $c=1$, a contradiction. This shows that $i=1+2^m$ and $a^c=a^{1+2^m}$.

Suppose that $j=2^{m}$, that is, $b^c=a^{2^{m}}b$. Using~\eqref{a},~\eqref{b},~\eqref{apower} and~\eqref{newv1}, we obtain $$v_1^{-1}=
(v_1^{-1})^c=
(v_1^b)^c=
v_1^{bc}=
v_1^{ca^{2^{m}}b}=
v_1^{a^{2^{m}}b}=(v_1^{-1})^b=v_1,$$ a contradiction. Thus $j=0$ and $b^c=b$. 

We are now ready to get a final contradiction. Note that
$$v_2^c=(v_1^a)^c=v_1^{ac}=v_1^{ca^{2^{m}}a}=v_1^{a^{2^{m}}a}=(v_1^{-1})^a=v_2^{-1},$$
and from this we similarly obtain
$$v_3^c=
(v_2^a)^c=
v_2^{ac}=
v_2^{ca^{2^{m}}a}=
(v_2^c)^{a^{2^{m}}a}=
(v_2^{-1})^{a^{2^{m}}a}=
v_2^a=v_3.$$

From~\eqref{J} and~\eqref{V}, we have $z=[v_1,v_2]$ and hence by applying  $c$ on both sides of this equality we get 
\begin{equation}\label{eqz}
z^c=[v_1,v_2]^c=[v_1^c,v_2^{c}]=[v_1,v_2^{-1}]=[v_1,v_2]^{-1}=z^{-1}.
\end{equation} Again from~\eqref{J} and~\eqref{V}, we have $z^{-1}=[v_1,v_3]$ and hence by applying $c$ on both sides of this equality we get
$$(z^{-1})^c=[v_1,v_3]^c=[v_1^c,v_3^c]=[v_1,v_3]=z^{-1}$$ and $z^c=z$. Clearly, this  contradicts~\eqref{eqz}.
 \end{proof}

\begin{proof}[Proof of Theorem~\ref{neg}]
For every $m\geq 1$, let $G_m$ be the group defined in~\eqref{G} and let $\Gamma_m$ be the graph defined in~\eqref{Gamma}. From Lemmas~\ref{1} and~\ref{4}, $\Gamma_m$ is a cubic vertex-transitive graph with $2^{m+1}3^{2^m+1}$ vertices and with automorphism group $G_m$.  From Lemma~\ref{semiregularG}, every semiregular automorphism of $\Gamma_m$ has order at most $6$. As $\lim_{m\to\infty}|V\Gamma_m|$, the theorem is proved.
 \end{proof}

\end{document}